\documentclass[letterpaper, 10pt]{article}
\usepackage{amsthm,amsmath,amsfonts,amssymb}

\newtheorem{thm}{Theorem}[section]
\newtheorem{our_thm}{Theorem}
\newtheorem{cor}[thm]{Corollary}
\newtheorem{lem}[thm]{Lemma}

\newtheorem{conj}[thm]{Conjecture}

\newtheorem{clm}[thm]{Claim}
\theoremstyle{definition}
\newtheorem{defn}[thm]{Definition}
\theoremstyle{remark}
\newtheorem{rem}[thm]{Remark}

\newcommand{\mP}{\mathcal{P}}
\newcommand{\mG}{\mathcal{G}}

\newcommand{\mB}{\mathcal{B}}
\newcommand{\mD}{\mathcal{D}}
\newcommand{\mA}{\mathcal{A}}

\newcommand{\mL}{\mathcal{L}}
\newcommand{\bd}{\mathbf{d}}
\newcommand{\wbd}{\widetilde{\bd}}
\newcommand{\bk}{\mathbf{k}}

\newcommand{\Gnp}[2]{\mG(#1,#2)}

\newcommand{\GNp}[1]{\Gnp{n}{#1}}
\newcommand{\GNP}{\GNp{p}}

\newcommand{\Hamiltonicity}{\mathcal{HAM}}
\newcommand{\Prob}[1]{\Pr\left[#1\right]}
\newcommand{\cProb}[2]{\Pr\left[ \left. #1 \;\right\vert #2 \right]}

\newcommand{\sExp}[2]{\mathbf{E}_{#1}\left[#2\right]}
\newcommand{\Exp}[1]{\sExp{}{#1}}
\newcommand{\Bin}[2]{\mathbf{Bin}\left(#1,#2\right)}

\title{On the resilience of Hamiltonicity and optimal packing of Hamilton cycles in random graphs}
\author{
Sonny Ben-Shimon\thanks{School of Computer Science, Raymond and Beverly Sackler Faculty of Exact Sciences, Tel Aviv University, Tel Aviv 69978, Israel. E-mail: sonny.benshimon@cs.tau.ac.il. Research partially supported by a Farajun Foundation Fellowship.}
\and Michael Krivelevich\thanks{School of Mathematical Sciences, Raymond and Beverly Sackler Faculty of Exact Sciences, Tel Aviv University, Tel Aviv 69978, Israel. E-mail: krivelev@tau.ac.il. Research supported in part by a USA-Israel BSF grant and by a grant from the Israel Science Foundation.}
\and Benny Sudakov\thanks{Department of Mathematics, UCLA, Los Angles 90005, CA, USA. Email: bsudakov@math.ucla.edu. Research supported in part by NSF CAREER award DMS-0812005 and by USA-Israeli BSF grant.}}
\date{}
\begin{document}
\maketitle
\begin{abstract}
Let $\bk=(k_1,\ldots,k_n)$ be a sequence of $n$ integers. For an increasing monotone graph property $\mP$ we say that a base graph $G=([n],E)$ is \emph{$\bk$-resilient} with respect to $\mP$ if for every subgraph $H\subseteq G$ such that $d_H(i)\leq k_i$ for every $1\leq i\leq n$ the graph $G-H$ possesses $\mP$. This notion naturally extends the idea of the \emph{local resilience} of graphs recently initiated by Sudakov and Vu. In this paper we study the $\bk$-resilience of a typical graph from $\GNP$ with respect to the Hamiltonicity property where we let $p$ range over all values for which the base graph is expected to be Hamiltonian. In particular, we prove that for every $\varepsilon>0$ and $p\geq\frac{\ln n+\ln\ln n +\omega(1)}{n}$ if a graph is sampled from $\GNP$ then with high probability removing from each vertex of ``small'' degree all incident edges but two and from any other vertex at most a $(\frac{1}{3}-\varepsilon)$-fraction of the incident edges will result in a Hamiltonian graph. 

Considering this generalized approach to the notion of resilience allows to establish several corollaries which improve on the best known bounds of Hamiltonicity related questions. It implies that for every positive $\varepsilon>0$ and large enough values of $K$, if $p>\frac{K\ln n}{n}$ then with high probability the local resilience of $\GNP$ with respect to being Hamiltonian is at least $(1-\varepsilon)np/3$, improving on the previous bound for this range of $p$. Another implication is a result on optimal packing of edge disjoint Hamilton cycles in a random graph. We prove that if $p\leq\frac{1.02\ln n}{n}$ then with high probability a graph $G$ sampled from $\GNP$ contains $\lfloor\frac{\delta(G)}{2}\rfloor$ edge disjoint Hamilton cycles, extending the previous range of $p$ for which this was known to hold.
\end{abstract}
\section{Introduction}\label{s:Introduction}
A \emph{Hamilton cycle} in a graph is a simple cycle that traverses all vertices of the graph. The study whether a graph contains a Hamilton cycle, or the \emph{Hamiltonicity} graph property, became one of the main themes of graph theory from the very beginning. Deciding whether a graph is \emph{Hamiltonian} was one of the first problems that were proved to be \textsc{NP-Complete}, giving some insight on the elusiveness of this problem. Although the computational problem of determining in ``reasonable'' time whether a graph is Hamiltonian seems hopeless the road does not need to stop there, as there are many other natural and interesting questions about this property. For one, given a prescribed number of edges $m$, are most graphs on $n$ vertices with $m$ edges Hamiltonian? Hence, the study of Hamiltonicity of \emph{random graphs} seems like a natural approach to pursue along the path to understand this intricate property. 

Fixing $m$ and selecting uniformly at random a graph on $n$ vertices with $m$ edges was indeed the original random graph model introduced by Erd\H{o}s and R\'{e}nyi, but the most widely studied random graph model is the binomial random graph, $\GNP$. In this model we start with $n$ vertices, labeled, say, by $1,\ldots,n$, and select a graph on these $n$ vertices by going over all $\binom{n}{2}$ pairs of vertices, deciding independently with probability $p$ for a pair to be an edge. The product probability space nature gives this model a greater appeal than the original one, but they are indeed very much related, and in a sense, equivalent (see monographs \cite{Bol2001} and \cite{JanLucRuc2000} for a thorough introduction to the subject of random graphs). We note that we will sometimes abuse the notation and use $\GNP$ to denote both the distribution on the graphs just described and a random sample from this distribution; which of the two should be clear from the context.  

We denote by $\Hamiltonicity$ the graph property of having a Hamilton cycle. One of the cornerstone results in the theory of random graphs is that of Bollob\'{a}s \cite{Bol83} and of Koml\'os and Szemer\'edi \cite{KomSze83} who proved that if $G\sim\GNP$ for $p\geq\frac{\ln n+\ln\ln n +\omega(1)}{n}$ (where $\omega(1)$ is any function tending to infinity with the number of vertices, $n$) then with high probability (or w.h.p. for brevity)\footnote{In this paper, we say that a sequence of events $\mA_n$ in a random graph model occurs w.h.p. if the probability of $\mA_n$ tends to $1$ as the number of vertices $n$ tends to infinity.} $G\in\Hamiltonicity$. It is fairly easy to show (see e.g. \cite[Chapter 3]{Bol2001}) that if $p\leq\frac{\ln n+\ln\ln n-\omega(1)}{n}$ then w.h.p. $G$ contains at least one vertex of degree smaller than $2$ (and hence is not Hamiltonian). Not only does this give a very precise range of $p$ for which the typical graph from $\GNP$ is Hamiltonian, this idea suggests that the bottleneck for $\Hamiltonicity$ (at least for small values of $p$) stems from the vertices of low degree.
\subsection{Types of resilience}
Let $\mP$ be a monotone increasing graph property (i.e. a family of graphs on the same vertex set which is closed under the addition of edges and isomorphism). Counting the minimal number of edges one needs to remove from a base graph $G$ in order to obtain a graph not in $\mP$ may, arguably, seem like one of the most natural questions to consider. Indeed, this notion, which is now commonly denoted by the \emph{global resilience} of $G$ with respect to $\mP$ (or the \emph{edit distance} of $G$ with respect to $\overline{\mP}$) is one of the fundamental questions in extremal combinatorics. This field can be traced back to the celebrated theorem of Tur\'{a}n \cite{Tur41} which states (in this terminology) that the complete graph $K_n$ on $n$ vertices has global resilience $\frac{n}{2}(\frac{n}{r}-1)$ (assuming $r$ divides $n$) with respect to containing a copy of $K_{r+1}$. 

For some graph properties of global nature, such as $\Hamiltonicity$ or being connected, the removal of all edges incident to a vertex of minimum degree is enough to destroy them, hence supplying a trivial upper bound on the global resilience of any graph with respect to these properties. For such properties the notion of global resilience does not seem to convey what one would expect from such a distance measure. One would like to gain some control on the amount of edges incident to a single vertex that can be removed. To pursue this approach, for a given base graph $G=([n],E)$ one would like to gain better understanding of the possible degree sequences of subgraphs $H$ of $G$ for which the graph $G-H$ possesses the property $\mP$.

Let $\mathbf{a}=(a_1,\ldots, a_n)$ and $\mathbf{b}=(b_1,\ldots, b_n)$ be two sequences of $n$ numbers. We write $\mathbf{a}\leq\mathbf{b}$ if $a_i\leq b_i$ for every $1\leq i\leq n$. Given a graph $G$ on vertex set $[n]$, we denote its degree sequence by $\bd_G=(d_G(1),\ldots,d_G(n))$. 
\begin{defn}
Let $G=([n],E)$ be a graph. Given a sequence $\bk=(k_1,\ldots,k_n)$ and a monotone increasing graph property $\mP$, we say that $G$ is $\bk$-resilient with respect to the property $\mP$ if for every subgraph $H\subseteq G$ such that $\bd_H\leq\bk$, we have $G-H\in\mP$.
\end{defn}
  
It was Sudakov and Vu \cite{SudVu2008} who initiated the systematic study of such a notion, albeit stated a little differently. In their original work the object of study was the minimum value of the maximum degree of a non $\bk$-resilient sequence. They coined this parameter as the \emph{local resilience} of a graph with respect to $\mP$. We will use the following notation to denote this parameter
$$r_\ell(G,\mP)=\min\{r:\exists H\subseteq G\;\mbox{such that}\;\Delta(H)=r\;\mbox{and}\; G\setminus H\notin \mP\}.$$
So, for local resilience, there is a uniform constraint on the number of deletions of edges incident to a single vertex. Although not explicitly, the study of local resilience lays in the heart of previous results in classical graph theory. In fact, one of the cornerstone results in the study of $\Hamiltonicity$ is the classical theorem of Dirac (see, e.g., \cite[Theorem 10.1.1]{Die2005}) which states (in this terminology) that $K_n$ has local resilience $n/2$ with respect to $\Hamiltonicity$. As has already been pointed out in \cite{SudVu2008} there seems to be a duality between the global or local nature of the graph property at hand and the type of resilience that is more natural to consider. More specifically, global resilience seems to be a more appropriate notion for studying \emph{local} properties (e.g. containing a copy of $K_k$), whereas for global properties (e.g. being Hamiltonian), the study of local resilience appears to be more natural. 

This study of resilience has gained popularity, and in a relatively short period of time quite a few research papers studied this and related distance notions \cite{SudVu2008, FriKri2008, DelEtAl2008, KriLeeSud2010, BalCsaSam2011, BenKriSudPre, AloSudPre, BotKohTarPre, LeeSamPre, HuaLeeSudPre,BalLeeSamPre}. This evolving body of research explored the resilience with respect to many graph properties, where the base graph of focus was mainly the binomial random graph $\GNP$ and graphs from families of pseudo-random graphs. In one of the subsequent papers of Dellamonica et. al. \cite{DelEtAl2008}, a more refined version of local resilience was considered. Consider a graph $G$ with degree sequence $\bd$. The authors of \cite{DelEtAl2008} defined the local resilience of $G$ with respect to $\mP$ as the maximal value of $0\leq \alpha\leq 1$ for which a graph $G$ $(\alpha\bd)$-resilient with respect to $\mP$. This notion is a little more robust than the original definition of local resilience as it can also deal with degree sequences of irregular graphs. The term $\bk$-resilience is a further generalization on the same theme. The main motivation for studying this more general notion is that it allows to give a unified approach to several problems as will be exposed.

\subsection{Previous work}\label{ss:PrevWork}
One of the first local resilience problems Sudakov and Vu \cite{SudVu2008} considered was the local resilience problem of $\GNP$ with respect to $\Hamiltonicity$. They provided both an upper and lower bound for this parameter for almost all the range of $p$. First, note that by the threshold probability for minimum degree $2$, the local resilience parameter becomes interesting only for $p\geq \frac{\ln n+\ln\ln n +\omega(1)}{n}$ (as for lower values of $p$, it is by definition equal to zero for non-Hamiltonian graphs, which w.h.p. is the case in this range). As an upper bound they proved that for every $0\leq p\leq1$ w.h.p. 
\begin{equation}\label{e:LocResHamLowBound}
r_\ell(\GNP,\Hamiltonicity)\leq\frac{np}{2}(1+o(1))
\end{equation}
Maybe more importantly, the lower bound they proved states that for every $\delta,\varepsilon>0$, if $p\geq\frac{\ln^{2+\delta} n}{n}$ then w.h.p. $r_\ell(\GNP,\Hamiltonicity)\geq\frac{np}{2}(1-\varepsilon)$ which essentially settles the problem for this range of $p$. Note that in fact this can be viewed as a far reaching generalization of Dirac's theorem. Frieze and Krivelevich in \cite{FriKri2008} studied this problem for the range of $p$ ``shortly after'' $\GNP$ becomes Hamiltonian w.h.p., but the lower bound they obtained in this range is weaker. They proved that there exist absolute constants $\alpha,C>0$ such that for every $p\geq\frac{C\ln n}{n}$ w.h.p. $r_\ell(\GNP,\Hamiltonicity)\geq\alpha np$. Recently, the authors in \cite{BenKriSudPre} improved on the above by showing that for every $\varepsilon>0$ there exists an absolute constant $C>0$ such that if $p\geq\frac{C\ln n}{n}$ then w.h.p. $r_\ell(\GNP,\Hamiltonicity)\geq\frac{np}{6}(1-\varepsilon)$. It is plausible that w.h.p. $r_\ell(\GNP,\Hamiltonicity)=\frac{np}{2}(1\pm o(1))$ as soon as $p\gg\frac{\log n}{n}$, but the above mentioned results still leave a gap to fill. In this work we make some progress on this front, but, alas, we are unable to close the gap completely.

A related question is the number of edge-disjoint Hamilton cycles one can have in a graph. Nash-Williams \cite{Nas71} asserted that Dirac's sufficient condition for a Hamilton cycle in fact guarantees at least $\lfloor\frac{5n}{224}\rfloor$ edge-disjoint Hamilton cycles. Quite recently, Christofides et. al. \cite{ChrEtAlPre}, answering one of Nash-Williams' original conjectures asymptotically, proved that minimum degree $\left(\frac{1}{2} + o(1)\right)n$ is sufficient for the existence of $\left\lfloor\frac{n}{8}\right\rfloor$ edge-disjoint Hamilton cycles in a graph. When considering random graphs, the current knowledge about packing of edge-disjoint Hamilton cycle is even more satisfactory. Bollob\'{a}s and Frieze \cite{BolFri85} showed that for every fixed $r$, if $p\geq\frac{\ln n+(2r-1)\ln\ln n+\omega(1)}{n}$, the minimal $p$ for which $\delta(\GNP)\geq 2r$, one can typically find $r$ edge-disjoint Hamilton cycles in $\GNP$. Kim and Wormald \cite{KimWor2001} established a similar result for random $d$-regular graphs (for fixed $d$), proving that such graphs typically contain $\lfloor\frac{d}{2}\rfloor$ edge-disjoint Hamilton cycles. The previous statements are of course best possible, but invite the natural question of what happens when the minimum degree is allowed to grow with $n$. Denote by $\mathcal{H}_\delta$ the property of a graph $G$ to contain $\lfloor \delta(G)/2\rfloor$-edge disjoint Hamilton cycles. Frieze and Krivelevich in \cite{FriKri2008}, extending \cite{BolFri85}, showed that if $p\leq\frac{(1+o(1))\ln n}{n}$ then w.h.p. $\GNP\in\mathcal{H}_\delta$. They even conjectured that this property is in fact typical for the whole range of $p$.
\begin{conj}[Frieze and Krivelevich \cite{FriKri2008}]\label{c:HamPackGnp}
For every $0\leq p(n)\leq 1$, w.h.p. $\GNP$ has the $\mathcal{H}_\delta$ property.
\end{conj}
In this paper, we are able to extend the range of $p$ for which Conjecture \ref{c:HamPackGnp} holds, but cannot resolve the conjecture completely.

\subsection{Our Results}\label{ss:OurResults}
As previously mentioned, in this work we explore the notion of $\bk$-resilience of random graphs with respect to $\Hamiltonicity$. To state our main result we need the following notation. Let $G=([n],E)$ be a graph. For every positive $t$ we denote by $\mD_t=\mD_t(G)=\{v\in [n]\;:\; d_G(v)<t\}$ the subset of vertices of degree less than $t$. Denote by $\bd=(d_1,\ldots, d_n)$ the degree sequence of $G$. For every constant $\varepsilon>0$ and $t>0$ we define the (not necessarily integral) sequence $\wbd(t,\varepsilon)=(\widetilde{d}_1,\ldots,\widetilde{d}_n)$ as follows: 
\begin{enumerate}
\item $\widetilde{d}_v=d_v-2$ for every $v\in \mD_t(G)$;
\item $\widetilde{d}_v=d_v\left(\frac{1}{3}-\varepsilon\right)$ for every $v\in V_1=V\setminus \mD_t(G)$.
\end{enumerate}

\begin{our_thm}\label{t:HamResGnp}
For every $\varepsilon>0$ and $p\geq\frac{\ln n+\ln\ln n+\omega(1)}{n}$ w.h.p. $G=([n],E)\sim\GNP$ with degree sequence $\bd$ is $\wbd(\frac{np}{100}, \varepsilon)$-resilient with respect to $\Hamiltonicity$.
\end{our_thm}
Note that Theorem \ref{t:HamResGnp} essentially covers the whole range of relevant values of $p$. Moreover, if $p$ is such that there exists some $\varepsilon>0$ for which w.h.p. $\delta(\GNP)-2\leq\frac{np}{100}\left(\frac{1}{3}-\varepsilon\right)$, then the result gives an exact local resilience for this range. On the one hand, every graph $G$ satisfies $r_\ell(G,\Hamiltonicity)\leq\delta(G)-1$ as in order to leave the graph Hamiltonian after the deletion of edges, all degrees in the resulting graph must be at least $2$. On the other hand, Theorem \ref{t:HamResGnp} guarantees that in this range of $p$ if $G\sim\GNP$ then w.h.p. removing any subgraph $H\subseteq G$ of maximum degree $\Delta(H)\leq\delta(G)-2$ leaves a Hamiltonian graph.
\begin{our_thm}\label{t:HamResGNPsmall}
If $\frac{\ln n+\ln\ln n+\omega(1)}{n}\leq p\leq\frac{1.02\ln n}{n}$ and $G\sim\GNP$ then w.h.p. $r_\ell(G,\Hamiltonicity)=\delta(G)-1$.
\end{our_thm}
\begin{proof}
As was previously mentioned, in light of Theorem \ref{t:HamResGnp} it is enough to prove that for the given range of $p$ w.h.p. $\delta(\GNP)-2\leq\frac{np}{100}\left(\frac{1}{3}-\varepsilon\right)$ for some $\varepsilon >0$. We use a basic result in the theory of random graphs due to Bollob\'{a}s (see e.g. \cite[Chapter 3]{Bol2001}) which asserts that if $\binom{n-1}{k}p^k(1-p)^{n-1-k}=\omega\left(\frac{1}{n}\right)$, then w.h.p. $\delta(\GNP)\leq k$. Note that if $k$ is such that w.h.p. $\delta(\GNP)\leq k$, then this is also true for every $p'\leq p$ due to monotonicity. Setting $k=\frac{\ln n}{300}$ it suffices to prove that for $p=\frac{1.02\ln n}{n-k}$, w.h.p. $\delta(\GNP)\leq k\leq\frac{np}{100}\left(\frac{1}{3}-\frac{1}{160}\right)+2$, which follows from
\begin{eqnarray*}
\binom{n-1}{k}p^k(1-p)^{n-1-k} &\geq& (1-o(1))\left(\frac{e(n-k)p}{k}\right)^k\cdot e^{-(n-k)p}\\
&\geq&(1-o(1))n^{\frac{1+\ln 306}{300}}\cdot n^{-1.02}\\
&=&\omega\left(\frac{1}{n}\right).
\end{eqnarray*}
This completes the proof of the theorem.
\end{proof}
Comparing Theorem \ref{t:HamResGNPsmall} with the upper bound from \eqref{e:LocResHamLowBound} we see that the local resilience drops from being equal to one less than the minimal degree in the beginning of the range to being equal to less than roughly half of it as $p$ grows. This is due to the fact that when $p$ is small enough the appearances of vertices whose degree is much smaller than the average degree create a bottleneck for $\Hamiltonicity$ (and many other graph properties).

Not all is lost when $p$ becomes larger. Our main result implies that taking $p$ to be large enough such that w.h.p. there are no vertices of degree less than $\frac{np}{100}$, the removal of almost one third of the incident edges at every vertex from a typical random graph will leave a graph which is Hamiltonian. Again, using straightforward calculations, which we omit, the following result is readily established.
\begin{our_thm}\label{t:HamLocResGNPbig}
For every $\varepsilon>0$ there exists a constant $C=C(\varepsilon)>0$ such that if $p\geq\frac{C\ln n}{n}$ then w.h.p. $$r_\ell(\GNP,\Hamiltonicity)\geq \frac{np}{3}\left(1-\varepsilon\right).$$
\end{our_thm}
Note that Theorem \ref{t:HamLocResGNPbig} in fact improves on the best known results of \cite{BenKriSudPre} for the local resilience of $\GNP$ with respect to $\Hamiltonicity$ in this range of $p$.

Lastly we show how from Theorem \ref{t:HamResGnp} one can obtain the existence of an optimal packing of Hamilton cycles in a typical random graph.
\begin{our_thm}\label{t:HamPackGnp}
For every $p\leq\frac{1.02\ln n}{n}$ w.h.p. $\GNP$ has the property $\mathcal{H}_\delta$.
\end{our_thm}
\begin{proof}
First note that in light of the result of Frieze and Krivelevich \cite{FriKri2008}, we may assume that $p\geq(1+o(1))\frac{\ln n}{n}$. We claim that in the range $(1+o(1))\frac{\ln n}{n}\leq p\leq \frac{1.02\ln n}{n}$ w.h.p. one can sequentially remove a Hamilton cycle from a typical graph $G\sim\GNP$ for $\lfloor\delta(G)/2\rfloor-1$ rounds leaving the graph Hamiltonian. Indeed, assume that the assertion of Theorem \ref{t:HamResGNPsmall} holds, then the removal of $0\leq i\leq\lfloor\delta(G)/2\rfloor-1$ edge-disjoint Hamilton cycles from $G$ is a removal of a $2i$-regular subgraph from $G$, and therefore Theorem \ref{t:HamResGNPsmall} asserts that the resulting graph must be Hamiltonian. The removal of the last Hamilton cycle concludes the proof.  
\end{proof}
Although the improvement of Theorem \ref{t:HamPackGnp} relative to the previous best known bound on $p$ of Frieze and Krivelevich \cite{FriKri2008}, may seem quite insignificant, it should be stressed that the method used in \cite{FriKri2008} cannot be made to work for $p=\left(1+\varepsilon\right)\frac{\ln n}{n}$ for any fixed $\varepsilon>0$, so the improvement presented here is more of a qualitative nature. Alas, the methods presented here too cannot be extended much further, as the degree sequence of the random graph becomes more balanced causing $\mD_{\frac{np}{100}}$ to be the empty set. So, Conjecture \ref{c:HamPackGnp} remains open.
\subsection{Organization}
The rest of the paper is organized as follows. We start with Section \ref{s:Preliminaries} where we state all the needed preliminaries that are used throughout the proofs of our results. Section \ref{s:Hamiltonicity} is devoted to showing why a graph with pseudorandom properties is in fact $\bk$-resilient to $\Hamiltonicity$, and Section \ref{s:Random} is dedicated to prove that all of the random-like properties needed in the previous section appear w.h.p. in $\GNP$. Section \ref{s:ProofHamResGnp} is devoted to the proof of the main result of this paper, namely, Theorem \ref{t:HamResGnp}, and we conclude the paper with some final remarks and open questions in Section \ref{s:Conclusion}

\section{Preliminaries}\label{s:Preliminaries}
In this section we provide the necessary background information needed in the course of the proofs of the main results of this paper. 
\subsection{Notation}\label{ss:Notation}
Given a graph $G=(V,E)$, the \emph{neighborhood} $N_G(U)$ of a subset $U\subseteq V$ of vertices is the set of vertices defined by $N_G(U)=\{v\notin U\;:\;\exists u\in U.\;\{v,u\}\in E\}$, and the degree of a vertex $v$ is $d_G(v)=|N_G(\{v\})|$. We denote by $E_G(U)$ the set of edges of $G$ that have both endpoints in $U$, and by $e_G(U)$ its cardinality. Similarly, for two disjoint subsets of vertices $U$ and $W$, $E_G(U,W)$ denotes the set of edges with an endpoint in $U$ and the other in $W$, and $e_G(U,W)$ its cardinality. We will sometime refer to $e_G(\{u\},W)$ by $d_G(u,W)$. We use the usual notation of $\Delta(G)$ and $\delta(G)$ to denote the respective maximum and minimum degrees in $G$. We say that $H$ is a \emph{spanning subgraph} of $G$ (or simply a subgraph, as all the subgraphs we consider will be spanning), and write $H\subseteq G$ if the graph $H=(V,F)$ has the same vertex set as $G$ and its edge set satisfies $F\subseteq E$. We will denote by $\ell(G)$ the length of a longest path in $G$.

Let $R<n$ be positive integers and $f:[R]\rightarrow\mathbb{R}^+$. We say that a graph $G=(V,E)$ on $n$ vertices is a $(R,f)$-\emph{expander} if every $U\subseteq V$ of cardinality $|U|\leq R$ satisfies $|N_G(U)|\geq f(|U|)\cdot|U|$. When $f$ is a constant function equal to some $\beta>0$ we say that $G$ is a $(R,\beta)$-expander. When a function $f:A\rightarrow\mathbb{R}^+$ satisfies $f(a)\geq c$ for any $a\in A$, where $c\geq 0$ is a constant, we simply write $f\geq c$.
\begin{rem}\label{r:supset_fkexpander}
Note that if $G=(V,E)$ is an $(R,f)$-expander, then every $H=(V,F)$ for $F\supseteq E$ is also an $(R,f)$-expander.
\end{rem}

The main research interest of this paper is the asymptotic behavior of some properties of graphs, when the graph is sampled from some probability measure $\mG$ over a set of graphs on the same vertex set $[n]$, and the number of vertices, $n$, grows to infinity. Therefore, from now on and throughout the rest of this work, when needed we will always assume $n$ to be large enough. We use the usual asymptotic notation. For two functions of $n$, $f(n)$ and $g(n)$, we denote $f=O(g)$ if there exists a constant $C>0$ such that $f(n)\leq C\cdot g(n)$ for large enough values of $n$; $f=o(g)$ or $f\ll g$ if $f/g\rightarrow 0$ as $n$ goes to infinity; $f=\Omega(g)$ if $g=O(f)$; $f=\Theta(g)$ if both $f=O(g)$ and $g=O(f)$.

Throughout the paper we will need to employ bounds on large deviations of random variables. We will mostly use the following well-known bound on the lower and the upper tails of the binomial distribution due to Chernoff (see e.g. \cite[Appendix A]{AloSpe2008}).
\begin{thm}[Chernoff bounds]\label{t:Chernoff}
Let $X\sim\Bin{n}{p}$, then for every $\alpha>0$
\begin{enumerate}
\item\label{i:Chernoff1} $\Prob{X > (1+\alpha)np}<\exp(-np((1+\alpha)\ln(1+\alpha)-\alpha))$;
\item\label{i:Chernoff2} $\Prob{X < (1-\alpha)np}<\exp(-\frac{\alpha^2np}{2})$; 
\end{enumerate}
\end{thm}

It will sometimes be more convenient to use the following bound on the upper tail of the binomial distribution.
\begin{lem}\label{l:BeckChe}
If $X \sim\Bin{n}{p}$ and $k \gg np$, then $\Prob{X \geq k} \leq \binom{n}{k}p^k\leq (enp/k)^k$.
\end{lem}

Lastly, we stress that throughout this paper we may omit floor and ceiling values when these are not crucial to avoid cumbersome exposition.

\section{From pseudorandomness to Hamiltonicity}\label{s:Hamiltonicity}
This section will provide all the necessary steps to show why a graph which possesses some random-like properties must be Hamiltonian. First, we show that a pseudorandom graph must contain an expander subgraph (Lemma \ref{l:GminHcontainsGamma}). We will also insist this expander subgraph is quite sparse in relation to the original graph, where the sparsity requirement on the expander subgraph will play a major role in the proof of Theorem \ref{t:HamResGnp}. As a second step we show how the expansion of small sets of vertices is a combinatorial property which is quite resourceful towards proving Hamiltonicity, where the main tool used to achieve this is the celebrated P\'{o}sa's rotation-extension technique (Lemma \ref{l:nEpsExpander}). 

\subsection{From pseudorandomness to expansion}
We start by defining a family of graphs with desired pseudorandom properties.
\renewcommand{\labelenumi}{\textbf{(P\arabic{enumi})}}
\begin{defn}\label{d:quasirand}
We say that a graph $G_1=(V_1,E_1)$ on $n_1$ vertices is $(n_1,d,\beta)$-\emph{quasi-random} if it satisfies the following properties:
\begin{enumerate}
\setcounter{enumi}{-1}
\item\label{i:quasidrand_mindeg} $\delta(G_1)\geq\frac{d}{150}$;
\item\label{i:quasidrand_smallsets} Every $U\subseteq V_1$ of cardinality $|U|<n_1^{0.11}\ln n_1$ satisfies $e_{G_1}(U)\leq d^{0.13}|U|$;
\item\label{i:quasidrand_medsets} Every $U\subseteq V_1$ of cardinality $|U|<12\beta n_1$ satisfies $e_{G_1}(U)\leq 50\beta d|U|$;
\item\label{i:quasidrand_bigsets} Every two disjoint subsets $U,Z\subseteq V_1$ where $|U|=\beta n_1$ and $|Z|=\left(\frac{1}{3}-27\beta\right)n_1$ satisfy $e_{G_1}(U,Z)\geq n_1\ln\ln n_1$.
\end{enumerate}
\end{defn}
The goal will be to find subgraphs of $(n_1,d,\beta)$-quasi-random graphs which satisfy certain expansion properties. To define this family of subgraphs we introduce the following ``expansion'' function. Given an integer $n_1$, we let $f_\beta:[\beta n_1]\rightarrow\mathbb{R}^{+}$ denote the function defined by: 
\renewcommand{\labelenumi}{\textbf{(Q\arabic{enumi})}}
\begin{enumerate}
\item\label{i:fbeta-expander_small} $f_\beta(t)=(\ln n_1)^{0.8}$ for every integer $1\leq t<n_1^{0.1}$;
\item\label{i:fbeta-expander_med} $f_\beta(t)=11$ for every integer $n_1^{0.1}\leq t< \beta n_1$;
\item\label{i:fbeta-expander_large} $f_\beta(\beta n_1)=\frac{2(1+39\beta)}{3\beta}$.
\end{enumerate}
The following lemma guarantees that every $(n_1,d,\beta)$-quasi-random graph contains a sparse $(\beta n_1,f_\beta)$-expander subgraph.
\begin{lem}\label{l:GminHcontainsGamma}
For every constant $0<\beta<\frac{1}{15\cdot10^4}$ there exists an integer $n_0=n_0(\beta)>0$ such that if $G_1$ is an $(n_1,d,\beta)$-quasi-random graph for some $n_1\geq n_0$ and $d\geq \ln n_1$ then $G_1$ contains a $(\beta n_1,f_\beta)$-expander subgraph $\Gamma$ satisfying $e(\Gamma)\leq 10^6\beta e(G_1)$.
\end{lem}
\begin{proof}
Pick every edge of $G_1=(V_1,E_1)$ to be an edge of $\Gamma$ with probability $\gamma=15\cdot10^4\beta$ independently of all other choices. Our goal is to prove that $\Gamma$ is an $(\beta n_1,f_\beta)$-expander with positive probability. 

First, we analyze the minimum degree of $\Gamma$. The degree of every vertex $v \in V_1$ in $\Gamma$ is binomially distributed, $d_{\Gamma}(v)\sim\Bin{d_{G_1}(v)}{\gamma}$, with median at least $\lfloor\gamma\delta(G_1)\rfloor$. Therefore $\Prob{d_{\Gamma}(v)\geq\lfloor\gamma\delta(G_1)\rfloor} \geq 1/2$. We choose $n_1$ to be large enough so that $\lfloor\gamma\delta(G_1)\rfloor\geq \frac{\gamma d}{200}$, then by property \textbf{P\ref{i:quasidrand_mindeg}} and the fact that the degrees in $\Gamma$ of every two vertices are positively correlated, we have that
$$\Prob{\delta(\Gamma)\geq\frac{\gamma d}{200}}\geq\Prob{\forall v\in V_1.\; d_\Gamma(v)\geq\lfloor\gamma\delta(G_1)\rfloor}\geq\prod_{v\in V_1}\Prob{d_\Gamma(v)\geq\lfloor\gamma\delta(G_1)\rfloor} \geq 2^{-n_1},$$
using the FKG inequality (see e.g. \cite[Chapter 6]{AloSpe2008}).

Under the assumption that $\delta(\Gamma)\geq\frac{\gamma d}{200}$ we show that $\Gamma$ must satisfy \textbf{Q\ref{i:fbeta-expander_small}}, namely that every $U\subseteq V_1$ of cardinality $|U|<n_1^{0.1}$ satisfies $|N_{\Gamma}(U)|\geq |U|(\ln n_1)^{0.8}$. Let $U\subseteq V_1$ be some subset of cardinality $|U|<n_1^{0.1}$, and assume that $|N_\Gamma(U)|<|U|(\ln n_1)^{0.8}$. Denote by $W=U\cup N_\Gamma(U)$, then by our assumption $|W|<|U|(1+(\ln n_1)^{0.8})\leq|U|(\ln n_1)^{0.81}\ll n_1^{0.11}\ln n_1$. Property \textbf{P\ref{i:quasidrand_smallsets}} of $G_1$ implies both $e_{\Gamma}(U)\leq e_{G_1}(U)\leq d^{0.13}|U|$ and $e_{\Gamma}(W)\leq e_{G_1}(W)\leq d^{0.13}|W|<|U|d^{0.13}(\ln n_1)^{0.81}$. On the other hand, $e_\Gamma(W)\geq\delta(\Gamma)\cdot|U|-e_{\Gamma}(U)\geq\frac{\gamma d}{200}\cdot|U|-e_{G_1}(U)>|U|d^{0.13}(\frac{\gamma(\ln n_1)^{0.87}}{200}-1)>|U|d^{0.13}(\ln n_1)^{0.81}$ which is a contradiction.

Property \textbf{Q\ref{i:fbeta-expander_med}} of $\Gamma$ will follow the exact same lines, again under the assumption that $\delta(\Gamma)\geq\frac{\gamma d}{200}$. Let $U\subseteq V_1$ be some subset of cardinality $n_1^{0.1}\leq|U|<\beta n_1$, and assume that $|N_\Gamma(U)|<11|U|$. Denote by $W=U\cup N_\Gamma(U)$, then by our assumption $|W|<12|U|<12\beta n_1$. Property \textbf{P\ref{i:quasidrand_medsets}} of $G_1$ implies both $e_{\Gamma}(U)\leq e_{G_1}(U)\leq 50\beta d|U|$ and $e_{\Gamma}(W)\leq e_{G_1}(W)\leq 50\beta d|W|<600\beta d|U|$. On the other hand, $e_\Gamma(W)\geq\delta(\Gamma)\cdot|U|-e_{\Gamma}(U)\geq\frac{\gamma d}{200}\cdot|U|-e_{G_1}(U)\geq750\beta d|U|-50\beta d|U|>600\beta d|U|$ which is a contradiction.

We proceed to show that $\Gamma$ satisfies the bound on its number of edges and property \textbf{Q\ref{i:fbeta-expander_large}} with probability greater than $1-2^{-n_1}$. This will imply that there exists a subgraph $\Gamma$ as stated by the lemma with positive probability. Consider the number of edges in $\Gamma$. Clearly, $e(\Gamma)\sim\Bin{e(G_1)}{\gamma}$ and by property \textbf{P\ref{i:quasidrand_mindeg}} we have $e(G_1)\geq\frac{n_1\delta(G_1)}{2}\geq\frac{n_1d}{300}\geq\frac{n_1\ln n_1}{300}$. Using Lemma \ref{l:BeckChe} it follows that
$$\Prob{e(\Gamma)>10^6\beta e(G_1)}\leq\left(\frac{e\cdot e(G_1)\cdot\gamma}{10^6\beta e(G_1)}\right)^{10^6\beta e(G_1)}=o(2^{-n_1}).$$

Fix a subset $U\subseteq V$ of cardinality $|U|=\beta n_1$. Assume that for every subset of vertices $Z\subseteq V_1$ of cardinality $|Z|=\left(\frac{1}{3}-27\beta\right)n_1$ and disjoint from $U$, we have that $e_\Gamma(U,Z)>0$, then $|N_\Gamma(U)|\geq\left(\frac{2}{3}+26\beta\right)n_1=\frac{2(1+39\beta)}{3\beta}|U|$. Therefore, in order to prove that $\Gamma$ satisfies \textbf{Q\ref{i:fbeta-expander_large}} with the required probability, we prove that $e_{\Gamma}(U,Z)>0$ for every pair of disjoint subsets of vertices $U,Z\subseteq V_1$ of cardinality $|U|=\beta n_1$, and $|Z|=\left(\frac{1}{3}-27\beta\right)n_1$ with probability at least $1-2^{-n_1}$. By property \textbf{P\ref{i:quasidrand_bigsets}} it follows that $e_{G_1}(U,Z)>n_1\ln\ln n_1$, and therefore $e_{\Gamma}(U,Z)$ is stochastically dominated by $\Bin{n_1\ln\ln n_1}{\gamma}$. This implies that
$$\Prob{e_{\Gamma}(U,Z)=0}<(1-\gamma)^{n_1\ln\ln n_1}<\exp(-\gamma n_1\ln\ln n_1).$$
Now, by going over all possible pairs of subsets $U$ and $Z$, as there are at most $4^{n_1}$ of those, we have that the probability that $\Gamma$ does not satisfy property \textbf{Q\ref{i:fbeta-expander_large}} is at most $\exp(-\gamma n_1\ln\ln n_1)\cdot4^{n_1}=o(2^{-n_1})$ which completes the proof.
\end{proof} 

\subsection{From expansion to Hamiltonicity: P\'{o}sa's rotation-extension technique}\label{ss:Posa}
In order to describe the relevant connection between Hamiltonicity and expanders, we require the notion of \emph{boosters}.
\begin{defn}\label{d:booster}
For every graph $G$ we say that a non-edge $\{u,v\}\notin E(G)$ is a \emph{booster} with respect to $G$ if $G+\{u,v\}$ is Hamiltonian or $\ell(G+\{u,v\})>\ell(G)$. For any vertex $v\in V$ we denote by
\begin{equation}
B_G(v)=\{w\notin N_G(v)\cup\{v\}\;:\;\{v,w\}\hbox{ is a booster}\}.
\end{equation}
\end{defn}
The following simple lemma describes a sufficient condition for a graph $G$ to be such that deleting the edges of some subgraph $H\subseteq G$ leaves a Hamiltonian graph. Note that a crucial point in the lemma is the existence of another graph which acts as a ``backbone'' by providing enough boosters.
\begin{lem}\label{l:boosters2Hamiltonian}
Let $G=(V,E)$ be a graph and let $H\subseteq G$ be some subgraph of it. If there exists a subgraph $\Gamma\subseteq G-H$ such that for every $E'\subseteq E(G-H-\Gamma)$ of cardinality $|E'|\leq n$ there exists a vertex $v\in V$ (which may depend on $E'$) satisfying $|N_{G-(\Gamma+E')}(v)\cap B_{\Gamma+E'}(v)|>d_H(v)$, then $G-H$ is Hamiltonian.
\end{lem}
\begin{proof}
Assume $\Gamma\subseteq G-H$ is as required by the lemma, then we prove there exists an edge set $F\subseteq E(G-H-\Gamma)$ such that the graph $\Gamma+F$ must be Hamiltonian, which in turn implies $G-H$ is Hamiltonian. Start with $F_0=\emptyset$. Assume that $F_i$ is a subset of $0\leq i\leq n$ edges of $E(G-H-\Gamma)$. If the graph $\Gamma_i=\Gamma+F_i\subseteq G-H$ is Hamiltonian we are done. Otherwise, by the assumption of the lemma, there exists a vertex $v_i\in V$ such that $|N_{G-\Gamma_i}(v_i)\cap B_{\Gamma_i}(v_i)|>d_H(v_i)$, and hence there exists at least one neighbor of $v_i$ in $G$, which we denote by $w_i$, such that the pair $\{v_i,w_i\}$ is still an edge in $G-H$, and is a booster with respect to $\Gamma_i$ (hence not an edge of $\Gamma_i$). It follows that the graph $\Gamma_i+\{v_i,w_i\}$ is Hamiltonian or $\ell(\Gamma_i+\{v_i,w_i\})>\ell(\Gamma_i)$. Finally, set $F_{i+1}=F_i\cup\{\{v_i,w_i\}\}$. Note that there must exist an integer $i_0\leq n$ such that $\Gamma_{i_0}$ is Hamiltonian, as the length of a path cannot exceed $n-1$.
\end{proof}

We now describe and apply a crucial technical tool, originally developed by P\'{o}sa \cite{Pos76}, which lies in the foundation of many Hamiltonicity results of random and pseudo-random graphs. This technique, which has come to be known as P\'{o}sa's rotation-extension, relies on the following basic operation on a longest path in a graph. We use the following two definitions.
\begin{defn}
Let $G=(V,E)$ be a graph, and let $P=(v_0,v_1,\ldots,v_\ell)$ be a longest path in $G$. If $\{v_i,v_\ell\}\in E$ for some $0\leq i\leq \ell-2$, then an \emph{elementary rotation} of $P$ along $\{v_i,v_\ell\}$ is the construction of a new longest path $P'=P-\{v_i,v_{i+1}\}+\{v_i,v_\ell\}=(v_0,v_1,\ldots,v_i,v_\ell,v_{\ell-1},\ldots,v_{i+1})$. We say that the edge $\{v_i,v_{i+1}\}$ is \emph{broken} by this rotation.
\end{defn}

Given a fixed $\beta>0$ and integer $n$, we define a family $\mL=\mL(n,\beta)$ of graphs with vertex set $V=[n]$ to which these elementary rotations will be applied to. The graphs in $\mL$ all entail some pseudo-random properties. Now, a graph $G=(V,E)\in\mL(n,\beta)$ if its vertex set can be partitioned $V=V_1\cup D$ such that:
\renewcommand{\labelenumi}{\textbf{(L\arabic{enumi})}}
\begin{enumerate}
\item\label{i:L-family-small-size} $|D|\leq n^{0.09}$;
\item\label{i:L-family-small-deg} $d(u,V_1)\geq 2$ for every $u\in D$;
\item\label{i:L-family-small-nopath} There is no path in $G$ of length at most $\frac{2\ln n}{3\ln\ln n}$ with both (possibly identical) endpoints in $D$;
\item\label{i:L-family-big} $G_1=G[V_1]$ is a $(\beta n_1,f_\beta)$-expander on $n_1=|V_1|$ vertices.
\end{enumerate}
Using elementary rotations we proceed to show that any $G\in\mL(n, \beta)$ must be Hamiltonian or that the subset of vertices with ``large'' $B_G(v)$ must also be large. Our proof uses similar ideas to those found in \cite{BenKriSudPre}.
\begin{lem}\label{l:nEpsExpander}
For every fixed $\beta>0$ there exists an integer $n_0(\beta)=n_0>0$ such that if $n\geq n_0$ then every graph $G=(V,E)\in\mL(n,\beta)$ is Hamiltonian or must satisfy $|\{v\in V_1\;:\; |B_G(v)|\geq n/3+\beta n\}|\geq n/3+\beta n$.
\end{lem}
\begin{rem}
We remark that property \textbf{L\ref{i:L-family-big}} implicitly implies an upper bound on $\beta$ for $\mL(n,\beta)$ to be non-empty. Indeed $\frac{2(1+39\beta)}{3}\leq 1-\beta$ must hold which implies $\beta\leq\frac{1}{81}$.
\end{rem}
\begin{proof}
First, we claim that under the assumptions on $G$, it must be connected. In order to prove this, we begin by showing that $G_1$ is connected. Indeed, assume otherwise and let $W\subseteq V_1$ be a connected component of cardinality $|W|\leq n_1/2$. Properties \textbf{Q\ref{i:fbeta-expander_small}} and \textbf{Q\ref{i:fbeta-expander_large}} imply that every subset of at most $\beta n_1$ vertices has a non-empty neighbor set, hence we can further assume that $|W|>\beta n_1$. Let $W'\subseteq W$ be of cardinality $\beta n_1$, then $|N_G(W')|>2n_1/3>n_1/2$ by property \textbf{Q\ref{i:fbeta-expander_large}} and hence cannot be contained in $W$, a contradiction. Property \textbf{L\ref{i:L-family-small-deg}} guarantees that every vertex in $D$ is connected to some vertex in $V_1$, so adding the vertices of $D$ to the graph $G_1$ leaves the graph connected. 

Take a longest path $P=(v_0,\ldots,v_{\ell})$ in $G$. By the assumption on $G$ we can clearly assume that $\ell\geq 3$. Since $P$ is a longest path, $N_G(v_0)\cup N_G(v_\ell)\subseteq P$. We first claim that we can choose such a path $P$ with both of its endpoints in $V_1$. So, assume that at least one of the endpoints of $P$, $v_0$ and $v_\ell$, is in $D$. If $v_0$ and $v_\ell$ are neighbors then $G$ must contain a cycle of length $\ell(G)$. This implies that $G$ is Hamiltonian, as otherwise $\ell(G)<n$ and since $G$ is connected there is an edge emitting out of this cycle creating a path of length $\ell(G)+1$ in $G$ which is a contradiction. We can thus assume $v_\ell$ and $v_0$ are not be neighbors. Assume w.l.o.g. $v_\ell\in D$ then by property \textbf{L\ref{i:L-family-small-deg}} the vertex $v_\ell$ must have a neighbor other than $v_{\ell-1}$. As all the neighbors of $v_\ell$ must lie on $P$ we denote this neighbor by $v_i$ where $1\leq i\leq \ell-2$. Performing an elementary rotation of $P$ along $\{v_i,v_\ell\}$ results in the path $P'=(v_0,v_1,\ldots,v_i,v_\ell,v_{\ell-1}\ldots,v_{i+1})$ with $v_{i+1}$ as an endpoint. By Property \textbf{L\ref{i:L-family-small-nopath}} of $D$, the vertex $v_{i+1}\notin D$ since it is at distance at most $2$ from $v_\ell\in D$. Denote the resulting longest path which terminates at $v_{i+1}$ by $P'$. If $v_0\in D$ as well then we can also perform a rotation on $P'$ keeping $v_{i+1}$ fixed. Note that in order to do so we can assume $v_0$ and $v_{i+1}$ are non-neighbors, as otherwise our graph in Hamiltonian as previously claimed. We can thus assume that our initial longest path $P$ has indeed two endpoints in $V_1$.

We set with foresight $r_0=\frac{\ln n_1}{7\ln\ln n_1}$, and $r$ to be the minimal integer for which $\beta n_1\leq 2^{r-r_0}\cdot T^{r_0}< 2\beta n_1$, where $T=\ln^{0.7}n_1$. For every $0\leq i\leq r$ we will find a set $S_i\subseteq V_1$ which is composed of (not necessarily all) endpoints of longest paths in $G$ obtained by performing a series of $i$ elementary rotations starting from $P$ while keeping the endpoint $v_0$ fixed such that for every $j<i$ in the $j$-th rotation the non-$v_0$ endpoint lies in $S_j\subseteq V_1$. We construct this sequence of sets $\{S_i\}_{i=0}^r$ as follows. We start with $S_0=\{v_\ell\}$ (which, by our assumption, is not in $D$). For every $i>0$, let $U_i=\{v_k\in N_{G}(S_i)\;:\;v_{k-1},v_k,v_{k+1}\notin\bigcup_{j=0}^i S_j\}$. Let $v_k\in U_i$, and $x\in S_i$ be such that $\{x,v_k\}\in E$, and denote by $Q$ the longest path from $v_0$ to $x$ obtained from $P$ by $i$ elementary rotations leaving $v_0$ fixed. By the definition of $U_i$, none of the vertices $v_{k-1},v_k,v_{k+1}$ is an endpoint of one of the sequence of longest paths starting from $P$ and yielding $Q$, hence both edges $\{v_{k-1},v_k\}$ and $\{v_k,v_{k+1}\}$ were not broken and are therefore present in $Q$. Rotating $Q$ along the edge $\{x,v_k\}$ will make one of the vertices $\{v_{k-1},v_{k+1}\}$ an endpoint in the resulting path. Assume w.l.o.g. that it is $v_{k-1}$, and hence add it to the set $S'_{i+1}$. Note that the vertex $v_{k-1}$ can also be added to the set $S'_{i+1}$ if the vertex $v_{k-2}$ in $U_i$, therefore
\begin{equation}\label{e:S_i}
|S'_{i+1}|\geq \frac{1}{2}|U_i|\geq \frac{1}{2}\left(|N_G(S_{i})|-3\sum_{j=0}^i|S_j|\right).
\end{equation}
Having defined the set $S'_{i+1}$ for every $i$, we proceed to demonstrate how the set $S_{i+1}$ is chosen as an appropriate subset of $S'_{i+1}$. This is done according to the value of $i$. First, for $0\leq i<r_0$ the set $S_i$ is chosen to be of cardinality $T^i$, where we prove we can do so inductively. Clearly, our assumption on $S_0$ proves that the base of the induction holds. Assuming we can choose the subsets $S_j\subseteq S'_{j}$ for all $j\leq i<r_0-1$, we prove we can choose a subset $S_{i+1}\subseteq S'_{i+1}\cap V_1$. We have by our induction hypothesis that $|S_i|=T^i<T^{r_0}\leq n_1^{0.1}$, hence for every $i<r_0$ property \textbf{Q\ref{i:fbeta-expander_small}} implies $|N_{G}(S_i)|\geq T^{i+1.1}$, and therefore by \eqref{e:S_i}
$$|S'_{i+1}|\geq\frac{1}{2}\left(T^{i+1.1}-3\left(\frac{T^{i+1}-1}{T-1}\right)\right)>T^{i+1}+1.$$ 
Every vertex in $S'_{i+1}$ is at distance at most $2i$ from $v_\ell$, therefore, every two vertices in $S'_{i+1}$ are at distance at most $4i<4r_0\leq\frac{4\ln n_1}{7\ln\ln n}<\frac{2\ln n}{3\ln\ln n}$ from each other. Property \textbf{L\ref{i:L-family-small-deg}} implies there can be at most one vertex from $D$ in this set. Removing this vertex if it exists, we can set $S_{i+1}$ to be any subset of $S'_{i+1}$ of $T^{i+1}$ vertices from $V_1$. In particular $|S_{r_0}|=T^{r_0}\gg n^{0.09}\geq |D|$. Next, for $r_0\leq i<r$ we will construct the sets $S_i$ to be of cardinality $R\cdot 2^{i-r_0}\gg |D|$, where $R=T^{r_0}$. Properties \textbf{Q\ref{i:fbeta-expander_small}} and \textbf{Q\ref{i:fbeta-expander_med}} imply $|N_{G}(S_i)|\geq 11\cdot|S_i|\geq 10|S_i|+2|D|$, and therefore by \eqref{e:S_i}
\begin{eqnarray*}
|S'_{i+1}|&\geq& \frac{1}{2}\left(10\cdot R\cdot2^{i-r_0}+2|D|-3\left(\sum_{j=0}^{r_0-1}T^i+\sum_{j=r_0}^{i}R\cdot 2^{j-r_0}\right)\right)\\
&\geq&|D|+5R2^{i-r_0}-\frac{3}{2}\left(T^{r_0}+R(2^{i+1-r_0}-1)\right)\\
&=& |D|+ 5R2^{i-r_0} - 3R2^{i-r_0} \geq R\cdot 2^{i-r_0} + |D|.
\end{eqnarray*}
We can therefore set $S_{i+1}$ to be any subset of $S'_{i+1}\cap V_1$ of cardinality $R\cdot 2^{i-r_0}$. For the sake of convenience we replace $S_r$ by some arbitrary subset of it of cardinality $\beta n_1$ and denote this new set by $S_r$. Finally, we construct similarly the set $S'_{r+1}$. By property \textbf{Q\ref{i:fbeta-expander_large}} we have $|N_G(S_r)|\geq(1+39\beta)2n_1/3$. Just as in the previous cases by \eqref{e:S_i}
\begin{eqnarray*}
|S'_{r+1}| &\geq& \frac{1}{2}\left(\frac{2n_1}{3}(1+39\beta)-3\left(\sum_{j=0}^{r_0-1}T^i+\sum_{j=r_0}^{r}R\cdot 2^{j-r_0}\right)\right)\\
&\geq&n_1\left(\frac{1}{3}+13\beta\right) - \frac{3}{2}T^{r_0}-\frac{3}{2}R\cdot2^{r-r_0+1}\\
&\geq&n_1\left(\frac{1}{3}+13\beta\right) - o(n_1) - 6\beta n_1\\
&>&\frac{n_1}{3}+6\beta n_1 > \frac{n}{3}+\beta n.
\end{eqnarray*}

Assume $S'_{r+1}\cap N_G(v_0)\neq \emptyset$, then $G$ must contain a cycle of length $\ell(G)$. This implies that $G$ is Hamiltonian. Assume otherwise0, then $\ell(G)<n$ and since $G$ is connected there is an edge emitting out of this cycle, creating a path of length $\ell(G)+1$ in $G$ which is a contradiction. This implies that $S'_{r+1}\subseteq B_G(v_0)$. Now take any endpoint $u_0$ in $S'_{r+1}$ and take a longest path $P'$ starting from $u_0$ (which must exist since all vertices of $S'_{r+1}$ are endpoints of longest paths starting in $v_0$) and repeat the same argument, while rotating $P'$ and keeping $u_0$ fixed. This way we obtain the desired set $|B_G(u_0)|$ of $n/3+\beta n$ endpoints for every $u_0\in S'_{r+1}$, thus completing the proof.
\end{proof}

\section{Structural properties of $\GNP$}\label{s:Random}
We start with a very simple claim regarding the number of edges in the binomial random graph model $\GNP$.
\begin{clm}\label{c:GNP_num_edges}
For every $p\geq\frac{\ln n}{n}$ w.h.p. $\frac{n^2p}{4}\leq e(\GNP)\leq n^2p$.
\end{clm}
\begin{proof}
This is a simple application of Theorem \ref{t:Chernoff}. Clearly, $e(G)\sim\Bin{\binom{n}{2}}{p}$ then $\Prob{e(G)>n^2p}<\Prob{e(G)\geq 2\binom{n}{2}\cdot p}\leq\exp(-\frac{(n-1)\ln n}{6})=o(1)$. Similarly, $\Prob{e(G)<\frac{n^2p}{4}}\leq\Prob{e(G)<0.51\binom{n}{2}p}\leq\exp\left(\frac{0.49^2(n-1)\ln n}{4}\right)=o(1)$.
\end{proof}

The following is (a special case of) a well known property of the binomial random graph model. It provides a very precise answer to which values of $p$ does a typical graph in $\GNP$ has minimum degree at least $2$ (see e.g. \cite{Bol2001}).
\begin{thm}\label{t:GNP_min_deg}
Let $h(n)=\omega(1)$ be any function which grows arbitrarily slowly with $n$, then
\begin{enumerate}
\renewcommand{\labelenumi}{\arabic{enumi})}
\item if $p\geq \frac{\ln n+\ln\ln n+h(n)}{n}$ then w.h.p. $\delta(\GNP)\geq 2$;
\item if $p\leq \frac{\ln n+\ln\ln n-h(n)}{n}$ then w.h.p. $\delta(\GNP)< 2$.
\end{enumerate}
\end{thm}
Recall that given a graph $G$ we defined the set $\mD_t(G)=\{v\in V\;:\; d_G(v)<t\}$. We prove some structural properties of the set $\mD_t(G)$ where $G$ is sampled from the random graph model $\GNP$.
\begin{clm}\label{c:Small-size}
For every $p\geq\frac{\ln n}{n}$ and integer $t\leq\frac{np}{100}$ w.h.p. $|\mD_t(\GNP)|\leq n^{0.09}$. 
\end{clm}
\begin{proof}   
Let $G\sim\GNP$, then setting $t_0=\frac{np}{100}$, we can bound the probability that a vertex is in $\mD_t(G)$ as follows. 
\begin{eqnarray*}
\Prob{d_G(v)<t}&\leq&\Prob{\Bin{n-1}{p}<t_0}\\
&\leq&\sum_{i=0}^{t_0-1}\binom{n-1}{i}p^{i}(1-p)^{n-1-i}\\
&\leq&t_0\cdot\binom{n-1}{t_0}p^{t_0}(1-p)^{n-1-t_0}\\
&\leq&t_0\cdot\left(\frac{e(n-1)p}{t_0}\right)^{t_0}e^{-p(n-1-t_0)}\\
&\leq&\exp\left(-np+p+\frac{np^2}{100}+\frac{np}{100}\left(1+\ln 100\right)+\ln\frac{np}{100}\right)\\
&<&e^{-0.92np}\\
&\leq&n^{-0.92}.
\end{eqnarray*}
This implies that $\Exp{\mD_t(\GNP)}\leq n^{0.08}$, and the claim follows from Markov's inequality.
\end{proof}

\begin{clm}\label{c:Small-nolongpath}   
For every $p\geq\frac{\ln n}{n}$ and integer $t\leq\frac{np}{100}$, w.h.p $G\sim\GNP$ does not contain a non-empty path of length at most $\frac{2\ln n}{3\ln\ln n}$ such that both of its (possibly identical) endpoints lie in $\mD_t(G)$.
\end{clm}
\begin{proof}
Setting $t_0=\frac{np}{100}$, we prove the claim for two distinct endpoints in $\mD_{t_0}(G)$, and for paths of length $r$ where $2\leq r\leq \frac{2\ln n}{3\ln \ln n}$. The other cases (i.e. identical endpoints or $r=1$) are similar and a little simpler. Fix two vertices $u,w\in V(G)$ and let $P = (u=v_0,\ldots,v_r=w)$ be a sequence of vertices of $V(G)$, where $2\leq r\leq \frac{2\ln n}{3\ln\ln n}$. Denote by $\mA_P$ the event $\{v_i, v_{i+1}\}\in E(G)$ for every $0 \leq i \leq r-1$, and by $\mB_{u,w}$ the event that both $u$ and $w$ are elements of $\mD_{t_0}(G)$. Clearly, $\Prob{\mA_P}=p^r$, then
\begin{equation*}
\Prob{\mB_{u,w}\wedge\mA_P}=p^r\cdot\cProb{\mB_{u,w}}{\mA_P}.
\end{equation*} 
Let $X_{u,w}$ denote the random variable which counts the number of edges in $G$ incident with $u$ or $w$ disregarding the pairs $\{u,v_1\}$, $\{v_{r-1},w\}$, and $\{u,w\}$. We can therefore bound $\cProb{\mB_{u,w}}{\mA_P}\leq\Prob{X_{u,w}<2t_0-2}$ and as $X_{u,w}\sim\Bin{2n-6}{p}$, it follows that  
\begin{eqnarray*}
\Prob{X_{u,w}<2t_0-2}&\leq&\sum_{i=0}^{2t_0-2}\binom{2n-6}{i}p^i(1-p)^{2n-6-i}\\
&\leq&\frac{np}{50}\binom{2n-6}{2t_0-2}p^{2t_0-2}(1-p)^{2n-4-2t_0}\\
&\leq&\frac{np}{50}\cdot\left(\frac{e(2n-6)p}{2t_0-2}\right)^{2t_0}e^{-p(2n-4-2t_0)}\\
&\leq&\exp\left(-2np+p\left(4+\frac{np}{50}\right)+\frac{np}{50}\left(2+\ln 100\right)+\ln\frac{np}{50}\right)\\
&<&e^{-1.8np}.
\end{eqnarray*} 
Fixing the two endpoints $u,w$, the number of such sequences is at most $(r-1)!\binom{n}{r-1}\leq n^{r-1}$. Applying a union bound argument over all such pairs of vertices and possible sequences connecting them we conclude that the probability there exists a path in $G$ of length $r \leq \frac{2\ln n}{3\ln\ln n}$, connecting two vertices of $\mD_t(G)$ is at most
\begin{eqnarray*}
\sum_{r=1}^{\frac{2\ln n}{3\ln\ln n}}\binom{n}{2}\cdot n^{r-1}\cdot p^r\cdot e^{-1.8np}&\leq&\sum_{r=1}^{\frac{2\ln n}{3\ln\ln n}}\frac{n^{r+1}}{2}\cdot\frac{\ln^r n}{n^r}\cdot n^{-1.8}\\
&\leq&\frac{\ln n}{3\ln\ln n}\cdot n^{-0.8}\cdot (\ln n)^{\frac{2\ln n}{3\ln\ln n}}\\
&=& o(1),
\end{eqnarray*}
where the first inequality follows by noting that the expression on l.h.s. decreases as $p$ grows, hence we can replace $p$ by $\frac{\ln n}{n}$. This completes the proof of the claim.
\end{proof}
Given a graph $G$ and some $t>0$ (which may depend on $n$ and $p$), we denote by $G_1(t)=G_1=(V_1,E_1)$ the graph $G_1=G[V\setminus \mD_t(G)]$ and by $n_1$ its number of vertices, i.e. $|V_1|=|V\setminus \mD_t(G)|=n_1$. The following lemma, which contains the main technical result of this section, implies that if $p\geq\frac{\ln n}{n}$ then the removal of the vertices of low degree from a typical graph $G$ sampled from $\GNP$ leaves a graph $G_1$ which is robust in the following sense: The deletion of almost a third of the edges at each vertex of $G_1$ leaves a graph with some strong pseudo-random properties.

\begin{lem}\label{l:GNP_r_quasirand}
For every fixed $\varepsilon>0$, there exists a small enough constant $\beta_0=\beta_0(\varepsilon)>0$ such that for every $0<\beta\leq\beta_0$, $p\geq\frac{\ln n}{n}$ and $t=\frac{np}{100}$, if $G\sim\GNP$ then w.h.p. for any subgraph $H_1\subseteq G_1(t)=G_1$ such that $\bd_{H_1}\leq (\frac{1}{3}-\varepsilon)\bd_{G}$, the graph $G_1-H_1$ is $(n_1,np,\beta)$-quasi-random.
\end{lem}
\begin{proof}
First, by Claim \ref{c:Small-nolongpath} we can assume that in $G$ every vertex of $V_1$ has at most one neighbor in $\mD_t(G)$, hence for every $v\in V_1$ we have $d_{G_1}(v)\geq d_G(v)-1$, and therefore $\delta(G_1)\geq t-1$. It follows that $\delta(G_1-H_1)>\left(\frac{2}{3}+\varepsilon\right)\delta(G_1)>\frac{2t}{3}=\frac{np}{150}$. Second, using Claim \ref{c:Small-size} we can, and will, assume that $n_1\geq n-n^{0.09}=n(1-o(1))$. The rest of the properties, and hence the proof of the lemma, will be a simple consequence from the following series of claims. We stress that throughout we will not compute $\beta$ explicitly, but we will assume it is small enough as a function of $\varepsilon$ for the arguments to go through.

\begin{clm}\label{c:GNP_small_sets}
W.h.p. every $U\subseteq V$ of cardinality $|U|\leq n^{0.11}\ln n$ satisfies $e_{G_1-H_1}(U)\leq e_{G}(U)\leq(np)^{3/25}|U|$.
\end{clm}
\begin{proof}   
Fixing such a subset of vertices $U$ of cardinality $u\leq n^{0.11}\ln n$, we have that $e_G(U)\sim\Bin{\binom{u}{2}}{p}$ and since $e\binom{u}{2}p\leq u(np)^{3/25}\cdot n^{-1/200}$ we have by Lemma \ref{l:BeckChe} that
\begin{eqnarray*}
\Prob{e_{G}(U)>(np)^{3/25}u}<\left(\frac{e\binom{u}{2}p}{(np)^{3/25}u}\right)^{(np)^{3/25}u}&\leq&\exp\left(-\frac{(np)^{3/25}\cdot u\cdot\ln n}{200}\right)\\
&\leq&\exp\left(-\frac{(\ln n)^{1.12}\cdot u}{200}\right).
\end{eqnarray*}
To upper bound the probability of the existence of a subset of vertices for which the assertion of the claim does not hold, we apply the union bound over all possible sets $U$ of cardinality $u\leq n^{0.11}\ln n$
\begin{eqnarray*}
\sum_{u=1}^{n^{0.11}\ln n}\binom{n}{u}\exp\left(-\frac{(\ln n)^{1.12}\cdot u}{200}\right)&\leq&\sum_{u=1}^{n^{0.11}\ln n}\exp\left(u\cdot\left(\ln\frac{en}{u}-\frac{(\ln n)^{1.12}}{200}\right)\right)\\
&\leq&n^{0.11}\ln n\cdot\exp\left(-(\ln n)^{1.1}\right)=o(1).
\end{eqnarray*}
\end{proof}

\begin{clm}
W.h.p. every $U\subseteq V$ of cardinality $|U|\leq 12\beta n$ satisfies $e_{G_1-H_1}(U)\leq e_{G}(U)\leq 50\beta np|U|$.
\end{clm}
\begin{proof}   
Very similarly to the proof of Claim \ref{c:GNP_small_sets} we fix a subset of vertices $U$ of cardinality $1\leq u\leq 12\beta n$, then noting that $e_G(U)\sim\Bin{\binom{u}{2}}{p}$ and since $\binom{u}{2}p\leq 50\beta npu$ we have by Lemma \ref{l:BeckChe} that
$$\Prob{e_{G}(U)>50\beta np u}<\left(\frac{e\binom{u}{2}p}{50\beta np u}\right)^{50\beta np u}\leq\exp\left(-50\beta npu\ln\frac{100\beta n}{eu}\right).$$ 
To upper bound the probability of the existence of a subset of vertices for which the assertion of the claim does not hold, we apply the union bound over all possible sets $U$ of cardinality $1\leq u\leq 12\beta n$
\begin{eqnarray*}
\sum_{u=1}^{12\beta n}\binom{n}{u}\exp\left(-50\beta np u\ln\frac{100\beta n}{eu}\right)&\leq&
\sum_{u=1}^{12\beta n}\exp\left(u\cdot\left(\ln\frac{en}{u}-50\beta np\ln\frac{100\beta n}{eu}\right)\right)\\
&\leq&\sum_{u=1}^{12\beta n}\exp\left(u\cdot\left(\ln\frac{n}{u}\left(1-50\beta \ln n\right) + 1 - 50\beta\ln n\ln\frac{100\beta}{e}\right)\right)\\
&=&o(1).
\end{eqnarray*}
\end{proof}

\begin{clm}
W.h.p. every two disjoint subsets $U,Z\subseteq V_1$ where $|U|=\beta n_1$ and $|Z|=n_1\left(\frac{1}{3}-27\beta\right)$ satisfy $e_{G_1-H_1}(U,Z)\geq n_1\ln\ln n_1$.
\end{clm}
\begin{proof}
Fix two such disjoint subsets of vertices $U,Z\subseteq V_1$ of the required cardinalities and note that 
$$e_{G_1-H_1}(U,Z) = e_{G_1}(U,Z)-e_{H_1}(U,Z)\geq e_{G}(U,Z)- \sum_{v\in U}d_{H_1}(v)\geq e_{G}(U,Z)- \left(\frac{1}{3}-\varepsilon\right)\sum_{v\in U}d_{G}(v).$$

As $e_G(U,Z)\sim\Bin{|U|\cdot|Z|}{p}$, we can apply Theorem \ref{t:Chernoff} item \ref{i:Chernoff2} and get that 
$$\Prob{e_G(U,Z)<\left(\frac{1}{3}-28\beta\right)n_1\cdot|U|\cdot p}\leq e^{-\Theta(n_1^2p)}=o(4^{-n_1}),$$
where the hidden constants in the exponent above are functions of $\beta$ alone. Next, we prove that the random variable $X(G)=\sum_{v\in U}d_{G}(v)$ is very likely not to deviate much from its expectation. To achieve this we resort to the Azuma-Hoeffding inequality for martingales of bounded variance (see e.g. \cite[Theorem 7.4.3]{AloSpe2008}). Note that for any two graphs on the vertex set $V$ that differ by a single edge, their value of $X$ can change by at most $2$, and $\Exp{X(G)}=|U|(n-1)p<\left(1+o(1)\right)|U|n_1p$. In order to divulge the value of $X$ one only needs to expose the pairs of vertices which have at least one endpoint in $U$. This implies that the total variance of the martingale is at most $\beta n_1^2p(1-p)$, and hence
$$\Prob{X(G)>(1+\beta)|U|n_1p}\leq e^{-\Theta(n_1^2p)}=o(4^{-n_1}),$$
where the hidden constants in the exponent above are, again, functions of $\beta$ alone. Recalling that $\beta$ can be made small enough with respect to $\varepsilon$ and that $|U|n_1p\gg n_1\ln\ln n_1$ we have that
$$\left(\frac{1}{3}-28\beta\right)>\left(\frac{1}{3}-\varepsilon\right)(1+\beta) + \frac{n_1\ln\ln n_1}{|U|n_1p},$$
and hence $\Prob{e_{G_1-H_1}(U,Z)<n_1\ln\ln n_1}=o(4^{-n_1})$. By applying the union bound over all pairs of subsets of vertices $U$ and $Z$, the proof of the claim in completed.
\end{proof}
Recalling Definition \ref{d:quasirand} completes the proof of the lemma.
\end{proof}

We can now present the main result of this section which based on the above is readily established. In what follows let $G=(V,E)$ with vertex set $V=[n]$ be sampled from $\GNP$ and let $t=\frac{np}{100}$. We recall that for every $\varepsilon>0$ if $\bd=(d_1,\ldots,d_n)$ is the degree sequence of $G$, then $\wbd(t,\varepsilon)=(\widetilde{d}_1,\ldots,\widetilde{d}_n)$ is the sequence defined by $\widetilde{d}_v=d_v-2$ for every $v\in \mD_t(G)$ and $\widetilde{d}_v=d_v(1/3-\varepsilon)$ for every $V_1=V\setminus\mD_t(G)$. 
\begin{cor}\label{c:GNP_contains_sparse_L}
For every fixed $\varepsilon>0$, there exists a small enough constant $\beta_0=\beta_0(\varepsilon)>0$ such that for every $0<\beta\leq\beta_0$ and $p\geq\frac{\ln n+\ln\ln n+\omega(1)}{n}$ the following holds. W.h.p. every subgraph $H\subseteq G$ of degree sequence $\bd_H\leq\wbd(t,\varepsilon)$ is such that the graph $G-H$ contains a subgraph $\Gamma_0\in\mL(n,\beta)$ which spans at most $2\cdot10^6\beta n^2p$ edges. Moreover, adding to $\Gamma_0$ any subset of edges $E_0\subseteq E$ results in a graph in $\mL(n,\beta)$ and the partition $V=\mD_t(G)\cup V_1$ guarantees that $\Gamma_0+E_0$ is in $\mL(n,\beta)$.
\end{cor}
\begin{proof}
Fix $\varepsilon>0$, let $\beta_0$ be as guaranteed by Lemma \ref{l:GNP_r_quasirand}, let $0<\beta\leq\beta_0$, and set $t=\frac{np}{100}$. Fix a subgraph $H\subseteq G$ with degree sequence $\bd_H\leq\wbd(t,\varepsilon)$, denote by $G_1=G_1(t)=(V_1,E_1)$, let $H_1=H[V_1]$, and set $n_1=|V_1|$. We can assume that $G$ satisfies the following properties:
\begin{enumerate}
\renewcommand{\labelenumi}{\arabic{enumi})}
\item $\delta(G-H)\geq 2$ (Theorem \ref{t:GNP_min_deg}).
\item $|\mD_t(G)|\leq n^{0.09}$ (Claim \ref{c:Small-size}).
\item There is no path of length at most $\frac{2\ln n}{3\ln\ln n}$ with both (possibly identical) endpoints in $|\mD_t(G)|$ (Claim \ref{c:Small-nolongpath}),
\item $G_1-H_1$ contains a $(\beta n_1,f_{\beta})$-expander spanning subgraph $\Gamma$ with at most $10^6\beta n^2p$ edges (Lemmata \ref{l:GNP_r_quasirand} and \ref{l:GminHcontainsGamma} where Claim \ref{c:GNP_num_edges} can be used to bound the number of edges in $e(G_1)$).
\end{enumerate}
To get the graph $\Gamma_0$, we add to the graph $\Gamma$ the set of vertices $D=\mD_t(G)$ with all of its incident edges from $G-H$. Note that $\Gamma_0\in\mL(n,\beta)$ (using the partition $V=V_1\cup D$) and that $e(\Gamma_0)=e(\Gamma)+e_{G-H}(D,V_1)\leq 10^6\beta n^2p + t\cdot n^{0.1}<2\cdot10^6\beta n^2p$ as claimed.

Consider any subset of edges $E_0\subseteq E(G)$. As $D$ is independent, all edges of $E_0$ which are not in $\Gamma_0$ must have at least one endpoint in $V_1$. The addition cannot create a path of length at most $\frac{2\ln n}{3\ln\ln n}$ between endpoints in $D$ since no such path exists in $G$. The addition of any edge from $E_G(D,V_1)$ to $\Gamma_0$ can only increase the degree of every vertex from $D$, and, finally, the addition of any edge from $E_G(V_1)$ to $\Gamma_0$ clearly leaves the induced subgraph on $V_1$ as a $(\beta n_1,f_{\beta})$-expander (as this is a monotone increasing graph property), and therefore the same partition of the vertex set $V=V_1\cup D$ also implies that $\Gamma_0+E_0\in\mL(n,\beta)$.
\end{proof}

\section{Proof of Theorem \ref{t:HamResGnp}}\label{s:ProofHamResGnp}
We can now provide the full proof of the main result of this paper, namely the proof of Theorem \ref{t:HamResGnp}. The road we take to achieve this is to show that given a typical graph $G$ sampled from $\GNP$, no matter how $H\subseteq G$ is chosen (given it satisfies the conditions on its degree sequence), not only will the graph $G-H$ contain a sparse expander subgraph $\Gamma$, it will also have as edges enough boosters with respect to $\Gamma$ as to transform it into a Hamiltonian graph.
\begin{proof}[Proof of Theorem \ref{t:HamResGnp}]
Fix $\varepsilon>0$, set $\beta=\beta(\varepsilon)$ to be a sufficiently small constant such that the assertion of Corollary \ref{c:GNP_contains_sparse_L} holds, and let $G\sim\GNP$ be with degree sequence $\bd$. If $G$ does not satisfy the conclusion of Corollary \ref{c:GNP_contains_sparse_L} we say that $G$ is \emph{corrupted}, and denote by $\mA_G$ this event of probability $o(1)$. 

Assume, then, that $G$ is not corrupted and \emph{not} $\wbd=\wbd(\frac{np}{100},\varepsilon)$-resilient to $\Hamiltonicity$, i.e. there exists a subgraph $H_0$ with degree sequence $\bd_H\leq\wbd$ for which $G-H_0$ is not Hamiltonian. Corollary \ref{c:GNP_contains_sparse_L} implies $G-H_0$ contains a subgraph $\Gamma_0\in\mL(n,\beta)$ which spans at most $2\cdot 10^6\beta n^2p$ edges, and moreover, adding to it any subset of edges from $E_1\subseteq E(G)$, results in a graph $\Gamma_0+E_1\in\mL(n,\beta)$. Lemma \ref{l:boosters2Hamiltonian} implies that there must exist a set $E_0\subseteq E(G-H)\subseteq E(G)$ of at most $n$ edges for which $|N_{G-(\Gamma_0+E_0)}(v)\cap\mB_{\Gamma_0+E_0}(v)|\leq\widetilde{d}_v$ for every vertex $v\in V$. As $|E_0|\ll e(\Gamma_0)$ from the above we conclude that $\Gamma_2=\Gamma_0+E_0\in\mL'(n,\beta)=\{\Gamma\in\mL(n,\beta)\;:\;|E(\Gamma)|\leq 10^7\beta n^2p\}$. 

Corollary \ref{c:GNP_contains_sparse_L} guarantees that $V=V_1\cup \mD_t(G)$ is a partition of the vertex set for which $\Gamma_2$ satisfies the properties of $\mL(n,\beta)$. Lemma \ref{l:nEpsExpander} implies the set $A=\{v\in V_1: |B_{\Gamma_2}(v)|\geq n(1/3+\beta)\}$ must satisfy $|A|\geq n(1/3+\beta)$. Let $A_0\subseteq A$ be a subset of cardinality $|A_0|=\frac{n}{3}$.

So, in fact, for non-corrupted $G$ we will resort to bound the event that there exists a graph $\Gamma_2\in\mL'$ that is contained in $G$ for which $|N_{G-\Gamma_2}(v)\cap\mB_{\Gamma_2}(v)|\leq\widetilde{d}_v=\left(\frac{1}{3}-\varepsilon\right)d_v$ for every vertex $v\in A_0$. One should note that from the independence of the appearance of edges in the $\GNP$ model, given some $\Gamma_2\in\mL'$ the events $\left[\Gamma_2\subseteq G\right]$ and $\left[|N_{G-\Gamma_2}(v)\cap\mB_{\Gamma_2}(v)|\leq\widetilde{d}_v\right]$ are independent as they stem from the appearance of disjoint sets of edges, i.e. $e(\Gamma_2)$ and $\binom{V}{2}\setminus e(\Gamma_2)$ respectively. Putting it all together yields that the probability that $G$ is not $\wbd$-resilient to $\Hamiltonicity$ is upper bounded by
\begin{eqnarray}
\nonumber&&\Prob{\mA_G} + \cProb{\exists\Gamma_2\in\mL'\;.\;\left(\Gamma_2\subseteq G\right) \wedge \left(\forall v\in V\;.\;|N_{G-\Gamma_2}(v)\cap\mB_{\Gamma_2}(v)|\leq\widetilde{d}_v\right)}{\overline{\mA_G}}\nonumber\\
\nonumber & \leq & o(1)+\frac{1}{\Prob{\overline{\mA_G}}}\sum_{\Gamma_2\in\mL'}\Prob{\Gamma_2\subseteq G}\cdot\Prob{\forall v\in V\;.\;|N_{G-\Gamma_2}(v)\cap\mB_{\Gamma_2}(v)|\leq\widetilde{d}_v}\nonumber\\
\nonumber&\leq & o(1)+(1+o(1))\sum_{\Gamma_2\in\mL'}p^{e(\Gamma_2)}\times\\
&&\Prob{\sum_{v\in A_0}|N_{G-\Gamma_2}(v)\cap\mB_{\Gamma_2}(v)|\leq\left(\frac{1}{3}-\varepsilon\right)\left(2e_G(A_0) + e_G(A_0,V\setminus A_0)\right)}.\label{e:ResHamUpBound}
\end{eqnarray}
Taking into account that we are using a union bound argument by summing over all graphs $\Gamma_2\in\mL'$ (and there may be an exponential number of those) it is left to show that every summand in the above expression is exponentially small. We define the following random variable (which depends on the choice of $\Gamma_2)$. Let 
$$X=X(G)=\sum_{v\in A_0}|N_{G-\Gamma_2}(v)\cap B_{\Gamma_2}(v)|,$$
whose expectation satisfies
$$\Exp{X}=\sum_{v\in A_0}\Exp{|N_{G-\Gamma_2}(v)\cap B_{\Gamma_2}(v)|}=p\cdot\sum_{v\in A_0}\Exp{|B_{\Gamma_2}(v)|}\geq\frac{n^2p}{9}.$$
Note that for any two graphs on the vertex set $V$ that differ by a single edge, their value of $X$ can change by at most $2$ ($1$ for every endpoint of the edge), hence we can apply the Azuma-Hoeffding inequality for martingales of bounded variance (see e.g. \cite[Theorem 7.4.3]{AloSpe2008}), to prove that $X$ is concentrated around its expectation. In the process of ``exposing'' the edges of the graph, it suffices to expose only the pairs with an endpoint in $A_0$ which are non-edges of $\Gamma_2$. This implies that the total variance of the martingale is upper bounded by $\frac{n^2}{3}p(1-p)$, and hence
\begin{equation}\label{e:Xuppbound}
\Prob{X(G)\leq\frac{n^2p}{9}(1-\varepsilon)}\leq \exp\left(-\Theta(\varepsilon^2n^2p)\right).
\end{equation}
On the other hand, a standard application of Theorem \ref{t:Chernoff} (the Chernoff bound) it follows that
\begin{equation}\label{e:NoDenseSet} 
\Prob{\exists U\subseteq V.\quad |U|=\frac{n}{3}\quad\wedge\quad 2e(U)+e(U,V\setminus U)>\frac{n^2p}{3}(1+\varepsilon)} = \exp\left(-\Theta(\varepsilon^2n^2p)\right).
\end{equation}
Using the fact that $\frac{n^2p}{9}(1-\varepsilon)\geq \frac{n^2p}{3}(\frac{1}{3}-\varepsilon)(1+\varepsilon)$, it follows from \eqref{e:Xuppbound} and \eqref{e:NoDenseSet} that
$$\Prob{X\leq\left(\frac{1}{3}-\varepsilon\right)\left(2e_G(A_0) + e_G(A_0,V\setminus A_0)\right)}\leq\exp\left(-\Theta(\varepsilon^2n^2p)\right).$$
Returning to our upper bound on the probability that $G$ is not $\wbd$-resilient to $\Hamiltonicity$, we set $\mu=10^7\beta$. Plugging in the above in \eqref{e:ResHamUpBound} the probability is upper bounded by
\begin{eqnarray*}
&& o(1) +(1+o(1))\sum_{m=1}^{\mu n^2p}\binom{\binom{n}{2}}{m}\cdot p^m\cdot\exp(-\Theta(\varepsilon^2n^2p))\\
&\leq & o(1) + (1+o(1))\sum_{m=1}^{\mu n^2p}\left(\frac{en^2p}{2m}\right)^m\cdot\exp(-\Theta(\varepsilon^2n^2p))\\
&\leq & o(1) + \exp(\Theta\left(\mu n^2p\ln\frac{1}{\mu}\right)-\Theta(\varepsilon^2n^2p)) = o(1),
\end{eqnarray*}
where from the second to the second inequality follows from the fact that $\left(\frac{en^2p}{m}\right)^m$ is increasing with $m$ for the given range, and that $\beta$ (and $\mu$) can be chosen small enough with respect to $\varepsilon$. This completes the proof of the theorem.
\end{proof}
\section{Concluding remarks and further research directions}\label{s:Conclusion}
This work is yet another building block in the recently initiated research area of resilience of graph properties. The generalized approach allowed us to tackle in a uniform way two different problems regarding Hamiltonicity. The main motivation for studying resilience of the type considered in this work is the ability to have refined control over vertices of small degree which create the main obstacle for Hamiltonicity when $p$ is in the low end of the range. As $p$ grows the degree sequence of the graph becomes more and more balanced, so this approach does not seem to be suitable (or even necessary) for larger values of $p$. Although in this work some progress has been made on the two fronts (the local resilience of random graphs with respect to $\Hamiltonicity$ and optimal packing of Hamilton cycles in random graphs), there are still gaps to fill in order to settle the two main questions considered in this paper, and we believe new ideas will be needed to resolve them completely.
\bibliographystyle{abbrv}
\bibliography{ImprovedLocRes}
\end{document}